\documentclass{amsart}
\usepackage{amssymb,graphpap, graphics}
\usepackage[vcentermath]{youngtab}
\usepackage{enumitem}
\usepackage{pstricks}

\newtheorem{theorem}{Theorem}[section]
\newtheorem{proposition}[theorem]{Proposition}
\newtheorem{corollary}[theorem]{Corollary}
\newtheorem{lemma}[theorem]{Lemma}

\theoremstyle{definition}

\theoremstyle{remark}

\newcommand{\R}{\mathbb{R}}
\newcommand{\Z}{\mathbb{Z}}

\begin{document}
\sloppy

\title[An Asymptotic Version of a Theorem of Knuth]
{An Asymptotic Version of a Theorem of Knuth}

\author[J. I. Novak]{Jonathan Novak}
\address{University of Waterloo, Faculty of Mathematics, Department of Combinatorics and Optimization, Waterloo, Ontario, Canada}
\email{j2novak@math.uwaterloo.ca}
\urladdr{www.math.uwaterloo.ca/~j2novak}
	

	\maketitle
	
	\section{Introduction}
	
	In this note we consider the asymptotics of the number $S(d,N)$ of permutations
	in the symmetric group $\mathfrak{S}(N)$ which have no decreasing subsequence
	of length $d+1,$ in the limit where $d\geq 2$ is a fixed but arbitrary positive integer
	and $N \rightarrow \infty.$  This is a fundamental problem in the subject of pattern
	avoidance in permutations, see \cite{AAM} and \cite[\S 7]{Stanley}.  In the interest of brevity, 
	familiarity with Young diagrams, Young tableaux, and 
	the Robinson-Schensted-Knuth (RSK) correspondence is assumed.  The reader is referred
	to Stanley's survey \cite{Stanley} for the necessary background and further references.  
	We adhere to 
	the notation and terminology of \cite{Stanley} save for the following exceptions:
	the $d \times q$ rectangular Young diagram is denoted $R(d,q)$ rather than
	$(q^d),$ and the number of standard Young tableaux of shape $\lambda$
	is denoted $\dim \lambda$ rather than $f^{\lambda}.$  Recall that the dimension
	of a Young diagram may be computed from Frobenius' fomula:
		\begin{equation}
			\label{Frobenius}
			\dim \lambda = \frac{\Gamma(\lambda_1+\dots+\lambda_d+1)}
			{\prod_{i=1}^d \Gamma(\lambda_i-i+d+1)} \prod_{1 \leq i<j \leq d}
			(\lambda_i-\lambda_j+j-i),
		\end{equation}
	where $d$ is any number such that $\lambda_{d+1}=0$ and $\Gamma(z)$ is the 
	gamma function.
	
	The following well-known exact formula for $S(2,N)$ is attributed to 
	Hammersley in \cite{Stanley}, with the first published proof due to Knuth 
	\cite[\S5.1.4]{Knuth}.
			
	\begin{theorem}[Knuth]
		Permutations with no decreasing subsequence of length $3$ are 
		counted by the Catalan numbers:
			$$S(2,N) = \dim R(2,N) = \frac{(2N)!}{N!(N+1)!}.$$
	\end{theorem}
	
	For $d > 2,$ there is no known closed formula for $S(d,N).$  The main result
	of this note is the following asymptotic version of Knuth's theorem.
		
	\begin{theorem}[Asymptotic Knuth theorem]
		\label{thm:main}
		For $d$ fixed and $n \rightarrow \infty,$
			$$S(d,dn) \sim \dim R(d,2n).$$
	\end{theorem}
	
	We will see below that $S(d,dn)>\dim R(d,2n)$ as soon as $d>2,$ so that
	Theorem \ref{thm:main} is false for $n$ finite.
	
	Via the RSK correspondence, an equivalent formulation of Theorem 
	\ref{thm:main} is the following.
	
	\begin{theorem}
		The number of permutations in $\mathfrak{S}(dn)$ with no decreasing
		subsequence of length $d+1$ is asymptotically equal, as 
		$n \rightarrow \infty,$ to the number of involutions in $\mathfrak{S}(2dn)$
		with longest decreasing subsequence of length exactly $d$ and 
		longest increasing subsequence of length exactly $2n.$
	\end{theorem}

	By Frobenius' formula, we have	
		\begin{equation}
			\label{rectangleExact}
			\dim R(d,q) = \frac{\Gamma(dq+1)}{\prod_{i=1}^{d} \frac{\Gamma(q+i)}
			{\Gamma(i)}}.
		\end{equation}
	Thus
		\begin{equation}
			\label{rectangleAsymptotic}
				\dim R(d,q) \sim (2\pi)^{\frac{1-d}{2}} \bigg{(} \prod_{i=1}^{d} \Gamma(i) 
				\bigg{)} d^{dq+\frac{1}{2}}q^{\frac{1-d^2}{2}}
		\end{equation}
	as $q \rightarrow \infty$ with $d$ fixed, by Stirling's formula.
	Setting $q=2N/d,$ Theorem \ref{thm:main} together with
	\eqref{rectangleAsymptotic} immediately implies the following.
	
	\begin{corollary}
		\label{cor:Regev}
		For $d$ fixed and $N \rightarrow \infty,$
			$$S(d,N) \sim (2\pi)^{\frac{1-d}{2}} \bigg{(} \prod_{i=1}^{d} \Gamma(i) 
				\bigg{)} d^{2N+\frac{d^2}{2}} (2N)^{\frac{1-d^2}{2}}.$$
	\end{corollary}
	
	Corollary \ref{cor:Regev} was first obtained by A. Regev \cite[Formula F.4.5.2]{R} in 1981 by 
	a rather different method, which will be discussed below.
		
	\section{Decomposition of rectangular tableaux}
	Given a Young diagram $\mu \subseteq R(d,q),$ let 
		\begin{equation}
			\mu^*=(q-\mu_d,\dots,q-\mu_1)
		\end{equation}
	denote the complement of $\mu$ relative to $R(d,q).$  Clearly,
		\begin{equation}
			\label{rectangleDecomp}
			\begin{split}
				\dim R(d,2n) &= \sum_{\substack{\mu \vdash dn\\ \mu \subseteq R(d,2n)}}
				(\dim \mu)(\dim \mu^*) \\
				&= \sum_{\substack{\mu \vdash dn\\ \mu \subseteq R(d,2n)\\
				\mu = \mu^*}}
				(\dim \mu)^2 + \sum_{\substack{\mu \vdash dn\\ \mu \subseteq R(d,2n)\\
				\mu \neq \mu^*}}
				(\dim \mu)(\dim \mu^*).
			\end{split}
		\end{equation}
	On the other hand, by RSK, we have
		\begin{equation}
			\label{RSKdecomp}
			\begin{split}
				S(d,dn) & =\sum_{\substack{\lambda \vdash dn\\ \ell(\lambda) \leq d}}
				(\dim \lambda)^2 \\
				&= \sum_{\substack{\mu \vdash dn\\ \mu \subseteq R(d,2n)\\
				\mu = \mu^*}}
				(\dim \mu)^2 + \sum_{\substack{\mu \vdash dn\\ \mu \subseteq R(d,2n)\\
				\mu \neq \mu^*}}
				(\dim \mu)^2 + \sum_{\substack{\nu \vdash dn\\ \nu_1>2n\\ \ell(\nu) \leq d}}
				(\dim \nu)^2.
			\end{split}
		\end{equation}
	Substituting for the first group of terms in \eqref{RSKdecomp} using \eqref{rectangleDecomp}
	and completing the square yields the following.
		
	\begin{proposition}
		$S(d,N)=\dim R(d,2n) + E(d,n),$ where the error term is given by
			$$E(d,n) = \frac{1}{2} \sum_{\substack{\mu \vdash dn\\ \mu \subseteq R(d,2n)}}
			(\dim \mu - \dim \mu^*)^2 +  \sum_{\substack{\nu \vdash dn\\ \nu_1>2n\\ \ell(\nu)\leq d}}
			(\dim \nu)^2.$$
	\end{proposition}
	
	Clearly $E(2,n)=0,$ in agreement with Knuth's theorem, while
	$E(d,n)>0$ for $d>2.$  Nevertheless, the error term is negligible in the limit
	$n \rightarrow \infty.$
	
	\section{Asymptotics of the dimension function}
	In order for $E(d,n)$ to be negligible, the sum $S(d,dn)$ must be dominated 
	in the $n \rightarrow \infty$ limit by self-complementary 
	diagrams contained in the rectangle $R(d,2n).$  
	The canonical self-complementary diagram relative to
	$R(d,2n)$ is the $d \times n$ rectangle $R(d,n).$  We consider
	the asymptotics of diagrams which deviate from $R(d,n)$ on the scale
	$\sqrt{n};$ this choice of scale emerges constructively in the proof of the 
	following key Lemma.
		
	\begin{lemma}
		\label{lem:key}
		For any distinct real numbers
			$$y_1> \dots > y_d$$
		satisfying
			$$y_1+\dots+y_d=0,$$
		we have
			$$\lim_{n \rightarrow \infty}C_{d,dn} \dim(n+y_1\sqrt{n},\dots,n+y_d\sqrt{n})
		= e^{-W(y_1,\dots,y_d)},$$
		where 
			$$C_{d,dn}=(2\pi)^{\frac{d}{2}} \frac{n^{dn+\frac{d(d+1)}{4}}}{\Gamma(dn+1)e^{dn}}$$
		and 
			$$W(y_1,\dots,y_d)=\frac{1}{2}\sum_{i=1}^d y_i^2 - \sum_{1 \leq i<j \leq d}
			\log(y_i-y_j).$$
	\end{lemma}
	
	\begin{proof}		
			Let $0 < \varepsilon<1,$ and consider deviations from $R(d,n)$ on the 
			scale $n^{\varepsilon}.$	
			By the Frobenius formula, we have
				$$\dim(n+y_1n^{\varepsilon},\dots,n+y_dn^{\varepsilon})=
				\frac{\Gamma(dn+1)}{\prod_{i=1}^d \Gamma(n+y_in^{\varepsilon}+d-i+1)}
					\prod_{1 \leq i<j \leq d} ((y_i-y_j)n^{\varepsilon}+j-i)$$
			for $n$ sufficiently large.
					
			Let us first analyze the asymptotics of the product
				$$\frac{1}{\prod_{i=1}^d \Gamma(n+y_in^{\varepsilon}+d-i+1)}.$$
			We begin by noting that $\Gamma(n+y_in^{\varepsilon}+d-i+1) \sim
			n^{d-i+1}\Gamma(n+y_in^{\varepsilon}),$ so that
				$$\frac{1}{\prod_{i=1}^d \Gamma(n+y_in^{\varepsilon}+d-i+1)}
				\sim \frac{1}{n^{\frac{d(d+1)}{2}}\prod_{i=1}^d \Gamma(n+y_in^{\varepsilon})}.$$
			Taking logarithms yields
				$$\log \bigg{(} \frac{1}{\prod_{i=1}^d \Gamma(n+y_in^{\varepsilon}+d-i+1)} 
				\bigg{)}
				\sim -\frac{d(d+1)}{2}\log n - \sum_{i=1}^d \log \Gamma(n+y_in^{\varepsilon}).$$
			Recall Stirling's formula:
				$$\log \Gamma(N) \sim \frac{1}{2}\log 2\pi + (N-\frac{1}{2})\log N -N$$
			for $N$ large.  Thus
				$$\sum_{i=1}^d \log \Gamma(n+y_in^{\varepsilon}) \sim
				\frac{d}{2}\log 2\pi - dn + \sum_{i=1}^d (n+y_in^{\varepsilon}-\frac{1}{2})
					\log(n+y_in^{\varepsilon}).$$
			Now since 
				$$\log(n+y_in^{\varepsilon})=\log n + \log(1+y_in^{\varepsilon-1}),$$
			we have
				$$\sum_{i=1}^d (n+y_in^{\varepsilon}-\frac{1}{2})\log(n+y_in^{\varepsilon})
				=dn \log n - \frac{d}{2} \log n + \sum_{i=1}^d (n+y_in^{\varepsilon}-\frac{1}{2})
				\log(1+y_in^{\varepsilon-1}).$$
			Using the expansion
				$$\log(1+y_in^{\varepsilon-1})=y_in^{\varepsilon-1}
				-\frac{1}{2}y_i^2n^{2\varepsilon-2}+O(n^{3\varepsilon-3})$$
			for $n$ sufficiently large, we find that
				$$\sum_{i=1}^d (n+y_in^{\varepsilon}-\frac{1}{2})
				\log(1+y_in^{\varepsilon-1}) \sim 
				\frac{n^{2\varepsilon-1}}{2}\sum_{i=1}^d y_i^2.$$
			Putting this all together, we find that
				$$\frac{1}{\prod_{i=1}^d \Gamma(n+y_in^{\varepsilon}+d-i+1)} \sim
				\frac{e^{dn}}{(2\pi)^{\frac{d}{2}}n^{dn+\frac{d^2}{2}}} 
				e^{-\frac{n^{2\varepsilon-1}}{2}\sum_{i=1}^d y_i^2}$$
			as $n \rightarrow \infty.$
			
			The second group of factors is much easier to handle:
				$$\prod_{1 \leq i<j \leq d} ((y_i-y_j)n^{\varepsilon}+j-i) \sim
				n^{\varepsilon\frac{d(d-1)}{2}} \prod_{1 \leq i< j \leq d} (y_i-y_j).$$
				
			Thus when $\varepsilon=\frac{1}{2}$ we have
				$$\dim (n+y_1\sqrt{n},\dots,n+y_d\sqrt{n}) \sim
				\frac{\Gamma(dn+1)e^{dn}}{(2\pi)^{\frac{d}{2}}n^{dn+\frac{d(d+1)}{4}}}
				e^{-W(y_1,\dots,y_d)}$$
			as $n \rightarrow \infty,$ as claimed.	
				
		\end{proof}
		
	\section{Riemann sum}
	Consider the generalized sum
		\begin{equation}
			S(d,dn;\beta)= \sum_{\substack{\lambda \vdash dn\\ \ell(\lambda) \leq d}}
			(\dim \lambda)^{\beta}, \quad \beta>0,
		\end{equation}
	and its presentation as a sum in parameters around the rectangle $R(d,n):$
		\begin{equation}
			S(d,dn;\beta) = \sum_{\substack{\frac{(d-1)n}{\sqrt{n}} \geq y_1 \geq
			\dots \geq y_{d-1} \geq \frac{-n}{\sqrt{n}}\\ y_i \in \frac{1}{\sqrt{n}}\Z}}
			\dim(n+y_1\sqrt{n},\dots,n+y_d\sqrt{n})^{\beta},
		\end{equation}
	where $y_d:=-(y_1+\dots+y_{d-1}).$ 
	The lattice $(\frac{1}{\sqrt{n}}\Z)^{d-1}$ partitions $\R^{d-1}$ into cells
		\begin{equation}
			\bigg{[} \frac{k_1}{\sqrt{n}}, \frac{k_1+1}{\sqrt{n}} \bigg{)} \times \dots \times
			\bigg{[} \frac{k_{d-1}}{\sqrt{n}}, \frac{k_{d-1}+1}{\sqrt{n}} \bigg{)}, \quad
			k_i \in \Z,
		\end{equation}
	of volume $(\frac{1}{\sqrt{n}})^{d-1}.$  Scaling 
	by the mesh volume makes this into a Riemann sum, and by Lemma \ref{lem:key} we have
		\begin{equation}
			\label{Regev}
			\lim_{n \rightarrow \infty} \bigg{(} \frac{1}{\sqrt{n}} \bigg{)}^{d-1} C_{d,dn}^{\beta}
			S(d,dn;\beta) = \int_{\Omega_{d-1}} e^{-\beta W(y_1,\dots,y_d)} dy,
		\end{equation}
	where $C_{d,dn}$ and $W$ are as in Lemma \ref{lem:key} and $\Omega_{d-1} \subset
	\R^{d-1}$ is the region
		\begin{equation}
			\Omega_{d-1} =\{(y_1,\dots,y_{d-1}) \in \R^{d-1} : y_1 > \dots > y_d:=-(y_1+\dots+
			y_{d-1})\}.
		\end{equation}
	The details of this convergence can be checked and made rigorous using the 
	dominated convergence theorem; we refer the reader to \cite{Matsumoto, R, Sniady}
	for the full argument.
	
	Consider now the sum
		\begin{equation}
			\begin{split}
				&\sum_{\substack{\mu \vdash dn\\ \mu \subseteq R(d,2n)}}
				(\dim \mu)^{\alpha}(\dim \mu^*)^{\beta-\alpha} \\ &= 
				\sum_{\substack{\frac{n}{\sqrt{n}} \geq y_1 \geq
				\dots \geq y_{d-1} \geq \frac{-n}{\sqrt{n}}\\ y_i \in \frac{1}{\sqrt{n}}\Z}}
				\dim(n+y_1\sqrt{n},\dots,n+y_d\sqrt{n})^{\alpha}
				\dim(n-y_d\sqrt{n},\dots,n-y_1\sqrt{n})^{\beta-\alpha},
			\end{split}
		\end{equation}
	where $0 \leq \alpha \leq \beta$ and as before we denote
	$y_d:=-(y_1+\dots+y_{d-1}).$  This can again be viewed as a Riemann sum,
	and by exactly the same argument we have
		\begin{equation}
			\begin{split}
			&\lim_{n \rightarrow \infty} \bigg{(} \frac{1}{\sqrt{n}} \bigg{)}^{d-1} C_{d,dn}^{\beta}
			\sum_{\substack{\mu \vdash dn\\ \mu \subseteq R(d,2n)}}
				(\dim \mu)^{\alpha}(\dim \mu^*)^{\beta-\alpha}\\ &=
				 \int_{\Omega_{d-1}} e^{-\alpha W(y_1,\dots,y_d)} e^{-(\beta-\alpha)
				 W(-y_d,\dots,-y_1)}dy.
			\end{split}
		\end{equation}
	
	\section{Symmetry}
	Note that the function $W$ has the symmetry
		\begin{equation}
			\label{symmetry}
			W(y_1,\dots,y_d)=W(-y_d,\dots,-y_1).
		\end{equation}
	It follows that 
		\begin{equation}
			 \int_{\Omega_{d-1}} e^{-\beta W(y_1,\dots,y_d)} dy=
			  \int_{\Omega_{d-1}} e^{-\alpha W(y_1,\dots,y_d)} e^{-(\beta-\alpha)
				 W(-y_d,\dots,-y_1)}dy,
		\end{equation}
	and thus we have proved the following.

	\begin{theorem}
		\label{thm:generalized}
		For any $\beta>0$ and $0 \leq \alpha \leq \beta,$ we have
			$$S(d,dn;\beta) \sim \sum_{\substack{\mu \vdash dn\\ \mu \subseteq R(d,2n)}}
			(\dim \mu)^{\alpha}(\dim \mu^*)^{\beta-\alpha}$$
		as $n \rightarrow \infty.$
	\end{theorem}
	
	Theorem \ref{thm:main} is the special case $\beta=2,\alpha=1$ of this 
	more general asymptotic equivalence.
	
	\section{Conclusion}
	The multidimensional integral 
		\begin{equation}
			\Psi(d;\beta)= \int_{\R^d} e^{-\frac{\beta}{2}\sum_{i=1}^d x_i^2}
			\prod_{1 \leq i<j \leq d} |x_i-x_j|^{\beta} dx, \quad \beta>0,
		\end{equation}
	is known as \emph{Mehta's integral}.  It is the partition function
	of a Coulomb gas of $d$ identical point charges $x_1>\dots>x_d$ on the real line
	at inverse temperature $\beta,$ with energy functional
		\begin{equation}
			W(x_1,\dots,x_d) = \frac{1}{2}\sum_{i=1}^d x_i^2 - \sum_{1 \leq i<j \leq d}
			\log(x_i-x_j).
		\end{equation}
	Dyson and Mehta \cite{DM} studied this integral and conjectured the formula
		\begin{equation}
			\label{DysonMehta}
				\Psi(d;\beta)=(2\pi)^{\frac{d}{2}} \beta^{-\frac{d}{2}-\beta\frac{d(d-1)}{4}}
				\prod_{i=1}^d \frac{\Gamma(1+i\frac{\beta}{2})}
				{\Gamma(1+\frac{\beta}{2})},
		\end{equation}
	which they verified for $\beta \in \{1,2,4\}$ using properties of Hermite polynomials
	(see e.g. \cite[\S3.5.1]{LZ} for the $\beta=2$ case of this argument).  Later, Bombieri
	observed that, for general $\beta,$ \eqref{DysonMehta} can be deduced 
	from the Selberg integral formula.  See \cite{FW} for the interesting history of this problem.
	
	Regev \cite[Lemma 4.3]{R} showed that 
		\begin{equation}
			\label{RegevLemma}
			\int_{\Omega_{d-1}} e^{-\beta W(y_1,\dots,y_d)}dy = \frac{1}{\Gamma(d+1)}
			\sqrt{\frac{\beta}{2\pi d}}\Psi(d;\beta),
		\end{equation}
	and used this fact together with equation \eqref{Regev} above to determine 
	the asymptotics of $S(d,N;\beta)$ from the known form \eqref{DysonMehta} of
	$\Psi(d;\beta).$  In this article, we have evaluated the asymptotics of 
	$S(d,N)=S(d,N;2)$ directly, without appealing to the exact value of 
	$\Psi(d;2).$  Thus, we may obtain 
		\begin{equation}
			\Psi(d;2) = (2\pi)^{\frac{d}{2}} 2^{-\frac{d^2}{2}}\prod_{i=1}^{d+1}\Gamma(i)
		\end{equation}
	by substituting the asymptotic form of $S(d,N)$ (Corollary \ref{cor:Regev})
	in \eqref{Regev} and using \eqref{RegevLemma}.  It would be interesting to know
	if the asymptotics of $S(d,N;\beta)$ can be determined directly in a similar way 
	for general $\beta >0.$  If so, this would yield a new and elementary verification of 
	\eqref{DysonMehta}.  In particular, for this purpose one may assume that $\beta$ is 
	an even integer, see \cite{DM}.  
			
	Finally, in this note we have only considered the asymptotics of $S(d,N)$ in the 
	single scaling limit where $N \rightarrow \infty$ with $d$ fixed.  Baik, Deift, and Johansson
	\cite{BDJ} have shown that 
		\begin{equation}
			S(d,N) \sim F(t) N!
		\end{equation}
	in the double scaling limit where $d,N \rightarrow \infty$ at the rate 
	$d \sim 2\sqrt{N} + tN^{1/6},$ with $t \in \R$ fixed.  
	Here $F(t)$ is the Tracy-Widom distribution function, 
	see \cite{BDJ}.  Since
		\begin{equation}
			S(d,dn)=\dim R(d,2n)+E(d,n),
		\end{equation}
	if the error term $E(d,n)$ can be effectively estimated in the double scaling limit 
	then concrete estimates for $F(t)$ will follow.
		
	\section{Acknowledgements}
	I would like to thank Sho Matsumoto and Andrei Okounkov for helpful correspondence,
	and Michael Albert for pointing out an error in an early version of this paper.

 \end{document}